\documentclass[11pt]{article}
\usepackage{latexsym,amsmath,amssymb}
\usepackage{url}
\usepackage{amsthm}
\makeatletter
\newtheorem*{rep@theorem}{\rep@title}
\newcommand{\newreptheorem}[2]{%
\newenvironment{rep#1}[1]{%
 \def\rep@title{#2 \ref{##1}}%
 \begin{rep@theorem}}%
 {\end{rep@theorem}}}
\makeatother

\newtheorem{theorem}{Theorem}
\newreptheorem{theorem}{Theorem}

\newtheorem{corollary}{Corollary}
\newtheorem{remark}{Remark}

\newtheorem{proposition}{Proposition}

\newcommand{\Z}{{\mathbb Z}}

\newcommand{\F}{{\mathbb F}}

\newcommand{\NC}{\textrm{NC}}

\newcommand{\ceil}[1]{\left \lceil #1 \right \rceil}

\newcommand{\card}[1]{\# \left( #1 \right)}

\newcommand{\rmv}[1]{}

\usepackage{xcolor}
\let\oldmarginpar\marginpar
\renewcommand\marginpar[1]{\-\oldmarginpar[\raggedleft\footnotesize\textcolor{blue}{#1}]%
{\raggedright\footnotesize\textcolor{blue}{#1}}}

\title{Counting Value Sets: Algorithm and Complexity}

\author{Qi Cheng\thanks{School of Computer Science,
The University of Oklahoma, Norman, OK 73019, USA. Email: {\tt
qcheng@cs.ou.edu}. Partially supported by NSF.} \and
Joshua E. Hill \thanks{ Department of Mathematics, University of
California, Irvine, CA 92697, USA. Email: {\tt hillje@math.uci.edu}.
Partially supported by NSF.}
\and Daqing Wan\thanks{ Department of Mathematics, University of
California, Irvine, CA 92697, USA. Email: {\tt dwan@math.uci.edu}.
Partially supported by NSF.} }

\begin{document}

\maketitle

\begin{abstract}
Let $p$ be a prime.  Given a polynomial in $\F_{p^m}[x]$ of degree
$d$ over the finite field $\F_{p^m}$, one can view it as a map from $\F_{p^m}$ to $\F_{p^m}$, and
examine the image of this map, also known as the value set.
In this paper,  we present the first non-trivial algorithm and the first complexity
result on computing the cardinality of
this value set.  We show an elementary connection between this
cardinality and the number of points on a family of varieties in
affine space.  We then apply Lauder and Wan's $p$-adic point-counting
algorithm to count these points, resulting in a non-trivial algorithm
for calculating the cardinality of the value set.  The running time of our algorithm
is $(pmd)^{O(d)}$.
In particular, this
is a polynomial time algorithm for fixed $d$ if $p$ is reasonably small.
We also
show that the problem is \#P-hard when the polynomial is given in a
sparse representation, $p=2$, and $m$ is allowed to vary, or when the
polynomial is given as a straight-line program, $m=1$ and $p$ is
allowed to vary.  Additionally, we prove that it is NP-hard to decide
whether a polynomial represented by a straight-line program has a root
in a prime-order finite field, thus resolving an open problem proposed by
Kaltofen and Koiran in \cite{Kaltofen03,KaltofenKo05}.
\end{abstract}

\newpage

\section{Introduction}

In a finite field with $q = p^m$ ($p$ prime) elements, $\F_q$, take a polynomial, $f \in \F_q \left[ x \right]$
with degree $d>0$.  Denote the image set of this polynomial as
\[ V_f = \left\{ f \left( \alpha \right) \mid \alpha \in \F_q \right\} \]
and denote the cardinality of this set as $\card{V_f}$.

There are a few trivial bounds that can be immediately established.
There are only $q$ elements in the field, so $\card{V_f} \leq q$.
Additionally, any polynomial of degree $d$ can have at most $d$ roots,
thus for all $a \in V_f$, $f(x) = a$ is satisfied at most $d$ times.
This is true for every element in $V_f$, so $\card{V_f} d \geq q$,
whence
\[ \ceil{\frac{q}{d}} \leq \card{V_f} \leq q \]
(where $\ceil{\cdot}$ is the ceiling function).

Both of these bounds can be achieved: if $\card{V_f} = q$, then $f$ is called a permutation polynomial and if $\card{V_f} = \ceil{\frac{q}{d}}$, then $f$ is said to have a ``minimal value set''.

The problem of establishing $\card{V_f}$ has been studied in various forms for at least the last 115 years, but exact formulations for $\card{V_f}$ are known only for polynomials in very specific forms.  Results that apply to general polynomials are asymptotic in nature, or provide estimates whose errors have reasonable bounds only on average \cite{MR1334624}.

The fundamental problem of counting the value set cardinality $\card{V_f}$
can be thought of as a much more general version of the problem of determining if a particular polynomial is a permutation polynomial.
Shparlinski  \cite{MR1190826} provided a baby-step giant-step type test that determines if a given polynomial is a permutation polynomial by extending \cite{MR1094533} to an algorithm that runs in $\tilde{O} ( (dq)^{6/7} )$. This is still fully exponential in $\log q$.
Ma and von zur Gathen \cite{MR1319494} provide
a ZPP (zero-error probabilistic polynomial time) algorithm for testing if a given polynomial is a permutation polynomial. According to \cite{Kayal05},
the first deterministic polynomial time algorithm for testing permutation polynomials is obtained by Lenstra using the classification of
exceptional polynomials which in turn depends on the classification of finite simple groups.  Subsequently, an elementary approach based on the
Gao-Kaltofen-Lauder factorization algorithm is given by Kayal \cite{Kayal05}.

For the more general problem of exactly computing $\card{V_f}$, essentially nothing is known about this problem's complexity and no non-trivial algorithms are known.  For instance, no baby-step giant-step type algorithm is known in computing $\card{V_f}$. No probabilistic polynomial time algorithm is
known. Finding a non-trivial algorithm and proving a non-trivial complexity result for the value counting were raised as open problems in \cite{MR1319494},
where a probabilistic approximation algorithm is given.
In this paper, we provide the first non-trivial algorithm and the first non-trivial complexity result for the exact counting of the
value set problem.

\subsection{Our results}

Perhaps the most obvious method to calculate $\card{V_f}$ is to
evaluate the polynomial at each point in $\F_q$ and count how
many distinct images result. This algorithm has a  time
and space complexity $ (d q)^{O(1)}  $.  One can also approach this problem by operating on points in the co-domain.  One has $f(x) = a$ for some $x \in \F_q$ if and only if $f_a(X) = f(X) - a$ has a zero in $\F_q$;
this algorithm again has a time complexity
$ (d q)^{O(1)}  $, but the space complexity is improved considerably
to $ (d\log q)^{O(1)}  $.

In this paper we present several results on determining the cardinality of
value sets. On the algorithmic side, we show an elementary connection between this
cardinality and the number of points on a family of varieties in
affine space.  We then apply the Lauder-Wan $p$-adic point-counting algorithm\cite{LauderWan2008countingpoints}, resulting in a non-trivial algorithm for calculating the image set
cardinality in the case that $p$ is sufficiently small (i.e., $p = O((d \log q)^C)$ for some positive constant $C$).  Precisely, we have

\begin{reptheorem}{thm:imagepointcounting}
There exists an explicit deterministic algorithm and an explicit polynomial $R$ such that for any $f \in \F_q [x]$ of degree $d$, where $q = p^m$ ($p$ prime), the algorithm computes the cardinality of the image set, $\card{V_f}$, in a number of bit operations bounded by $R \left( m^d d^d p^d \right)$.
\end{reptheorem}

The
running time of this algorithm is polynomial in both $p$ and $m$, but is
exponential in $d$. In particular, this is a polynomial time algorithm for fixed $d$ if the characteristic $p$
is small ($q=p^m$ can be large).

On the complexity side, we have several hardness results on the value set problem.
With a field of characteristic $p=2$, we have

\begin{reptheorem}{thm:binaryfieldvaluesetcounting}
The problem
of counting the value set of a sparse polynomial over a finite field of characteristic $p=2$
is \#P-hard.
\end{reptheorem}

The idea of our proof of this theorem is to reduce the problem of counting
satisfying assignments for a 3SAT formula to the problem of value set counting.

Over a prime-order finite field, we have

\begin{reptheorem}{thm:primeorderfieldvaluesetcounting}
Over a prime-order finite field $\F_p$,  the problem of counting the value set is \#P-hard under RP-reduction,
if the polynomial is given as a straight-line program.
\end{reptheorem}

Additionally, we prove that it is NP-hard to decide whether
a polynomial in $\Z[x]$ represented by a straight-line program has
a root in a prime-order finite field, thus
resolving an  open problem proposed in \cite{Kaltofen03,KaltofenKo05}.
We accomplish the complexity results over prime-order finite fields
by reducing the prime-order finite field subset sum problem (\textsc{PFFSSP})
to these problems.

In the \textsc{PFFSSP},
given a prime $p$, an integer $b$ and
a set of integers $S = \{a_1, a_2, \cdots, a_t \}$,
we want to decide the solvability of the equation
$$ a_1 x_1 + a_2 x_2 + \cdots + a_t x_t \equiv b  \pmod{p}$$
with $x_i \in \{ 0, 1\}$ for $1\leq i\leq t$.
The main idea comes from the observation that
if $t < \log p/3$,  there is
a sparse polynomial $\alpha (x) \in \F_p[x] $ such that
as $x$ runs over $\F_p$,
the vector
$$ (\alpha(x), \alpha (x+1), \cdots, \alpha (x+t-1))  $$
runs over all the elements in $\{0, 1\}^t  $. In fact, a lightly modified version of the quadratic character
$\alpha(x) = (x^{(p-1)/2} + x^{p-1})/2$ suffices. So
the \textsc{PFFSSP} can be reduced to deciding whether the shift sparse polynomial
$ \sum_{i=0}^{t-1} a_{i+1} \alpha (x+i) - b = 0  $
has a solution in $\F_p$.

\section{Background}

\subsection{The subset sum problem}

To prove the complexity results, we use the subset sum problem (\textsc{SSP}) extensively.
The \textsc{SSP} is a well-known problem in computer science.
In one instance of the \textsc{SSP}, given an integer $b$ and a set of positive integers
$S = \{a_1, a_2, \cdots, a_t \}  $,
\begin{enumerate}
\item (Decision version) the goal is to decide whether
there exists a subset $T\subseteq S $
such that the sum of all the integers in $T$ equals $b$,
\item (Search version)
the goal is to find a subset $T\subseteq S $
such that the sum of all the integers in $T$ equals $b$,
\item (Counting version)
the goal is to count the number of subsets $T\subseteq S $
such that the sum of all the integers in $T$ equals $b$.
\end{enumerate}
The decision version of  the \textsc{SSP} is a classical NP-complete problem.
The counting version of the \textsc{SSP} is \#P-complete,
which can be easily derived from proofs of the NP-completeness
of the decision version, e.g. \cite[Theorem 34.15]{CormenLe01}.

One can view the \textsc{SSP} as a problem of solving the linear equation
          $$ a_1 x_1 + a_2 x_2 + \cdots + a_t x_t = b  $$
with $x_i \in \{ 0, 1\}$ for $1\leq i\leq t$.
The prime-order finite field subset sum problem is a similar problem
where in addition to $b$
and $S$, one is given a prime $p$, and the goal is to decide
the solvability of the equation
$$ a_1 x_1 + a_2 x_2 + \cdots + a_t x_t \equiv b  \pmod{p}$$
with $x_i \in \{ 0, 1\}$ for $1\leq i\leq t$.

\begin{proposition}
The prime-order finite field subset sum problem is NP-hard under RP-reduction.
\end{proposition}

\begin{proof}
To reduce the subset sum problem to the prime-order finite field subset sum problem,
one finds a prime $p > \sum_{i=1}^t a_i $, which can be done
in randomized polynomial time.
\end{proof}

\begin{remark}
To make the reduction deterministic,
one needs to de-randomize the problem of finding a large prime,
which appears to be  hard \cite{TaoCr11}.
\end{remark}

\subsection{Polynomial representations}

There are different ways to represent a polynomial over a field $\F$.
The dense representation lists all the coefficients
of a polynomial, including the zero coefficients.
The sparse representation lists only the nonzero coefficients,
along with the degrees of the corresponding terms.
If most of the coefficients of a polynomial are zero,
then the sparse representation is much shorter than the
dense representation.
A sparse shift representation of a polynomial
in $\F[x]$ is a list of $n$ triples
$(a_i, b_i, e_i) \in \F \times \F \times \Z_{\geq 0}$ which
represents the polynomial $$ \sum_{1\leq i \leq n} a_i (x + b_i)^{e_i}. $$

More generally, a straight-line program for a univariate polynomial
in $\Z[x]$ or $\F_p [x]$ is
a sequence of assignments,
starting from $x_1 = 1$ and $x_2 = x$.
After that, the $i$-th assignment has the form
$$x_i = x_j \odot x_k$$
where $0\leq j, k < i$
and $\odot$ is one of the three operations $+, -, \times$.
We first let $\alpha$ be an element in $\F_{p^m}$ such that
$\F_{p^m} = \F_p [\alpha]$.
A straight-line program for a univariate polynomial
in $\F_{p^m} [x]$ can be defined similarly, except that
the sequence starts from
$x_1 = \alpha $ and $x_2 = x$.
One can verify that a straight-line program computes a univariate
polynomial,
and that sparse polynomials and sparse shift polynomials
have short straight-line programs.
A polynomial produced by a short straight-line program
may have very high degree,
and most of its coefficients may be nonzero,
so it may be costly to write it in either a dense form
or a sparse form.

\section{Hardness of solving straight-line polynomials}

It is known that deciding whether there is
a root in a finite field extension
for a sparse polynomial is NP-hard \cite{KipnisSh99}.
In a related work, it was shown that deciding whether there is
a $p$-adic rational root  for a sparse polynomial is
NP-hard \cite{AvendanoIb10}.
However, the complexity of deciding the solvability
of a straight-line polynomial in $\Z[x]$ 
within a prime-order finite field 
was not known.
This open problem was proposed in \cite{Kaltofen03} and
\cite{KaltofenKo05}.
We resolve this problem within this section, and this same idea will be used
later on to prove the hardness result of the value set
counting problem.


Let $p$ be an odd prime. Let $\chi$ be the quadratic character modulo $p$,
namely $\chi(x) $ equals $1, -1, \text{ or } 0$, depending on whether $x$
is a quadratic residue, a quadratic non-residue, or is congruent to $0$ modulo $p$.
For $x\in \F_p$, $\chi (x)= x^{(p-1)/2}$. Consider the list
\begin{equation}\label{complete}
 \chi(1), \chi(2), \cdots, \chi(p-1).
\end{equation}
It is a sequence in $\{1,-1\}^{p-1}$.
The following bound is a standard consequence of the celebrated Weil bound for character sums, see
\cite{peralta92} for a detailed proof.

\begin{proposition}
Let $(b_1, b_2, \cdots, b_t)$ be a sequence in $\{ 1,-1 \}^t$. Then
the number of $x \in \F_p$ such that
$$ \chi (x) =b_1, \chi(x+1) = b_2, \cdots, \chi(x+t-1) = b_t  $$
is in the range $ ( p/2^t - t(3+\sqrt{p}), p/2^t + t(3+\sqrt{p}))$.
\end{proposition}

The proposition implies that if $t < \log p /3$,
then every possible sequence in $\{-1, 1\}^t$
 occurs  as a consecutive sub-sequence in expression (\ref{complete}).
In many situations it is more convenient to use binary $0/1$ sequences,
which suggests instead using the polynomial $ (x^{(p-1)/2} + 1)/2 $,
but this results in a small problem at $x=0$.
We instead use the sparse polynomial $$ \alpha(x) = (x^{(p-1)/2} + x^{p-1})/2. $$
$\alpha (x)$
takes value in $\{0,1\}$ if $x\in \F_p$
and  $\alpha (x) =  1$ iff $\chi(x) = 1$.

\begin{corollary}\label{fromvectortoelement}
If $t < \log p/3$, then for any binary sequence
$ (b_1,b_2, \cdots, b_t) \in \{0,1\}^t$,
there exists a $x\in \F_p$ such that
$$\alpha(x) = b_1, \alpha(x+1) = b_2, \cdots, \alpha(x+t-1) = b_t. $$
\end{corollary}

In other words, if $t < \log p/3$, the map
$$x \mapsto (\alpha(x), \alpha(x+1), \cdots, \alpha(x+t-1))$$
is an {\em onto} map from $\F_p$ to $\{0, 1\}^t  $; this map thus sends an algebraic object to a
combinatorial object.

Given a straight-line polynomial $f(x) \in \Z[x]$ and a prime $p$,
how hard is it to decide whether the polynomial has a solution in $\F_p$?
We now prove that this problem is NP-hard.

\begin{theorem}
Given a sparse shift polynomial
$f(x) \in \Z[x]$, and a large prime $p$, it is NP-hard to decide whether
$f(x)$ has a root in $\F_p$.
\end{theorem}

\begin{proof}
We reduce the (decision version of the) subset sum problem to this problem.
Given
$b \in \Z_{\geq 0}$ and $S =  \{a_1, a_2, \cdots, a_t \} \subseteq \Z_{\geq 0}$,
one finds a prime
$p > \max(2^{3t}, \sum_{i=1}^t a_i) $ and constructs a
sparse shift polynomial
\begin{equation}
 \beta(x) = \sum_{i =0}^{t-1} a_i \alpha (x+i) - b.
\end{equation}
If the polynomial has a solution modulo $p$,
then the answer to the subset sum problem is ``yes'',
since for any $x\in \F_p$, $\alpha (x+i) \in \{0,1\}$.

In the other direction, if the answer to the  subset sum problem
is ``yes'', then according to Corollary~\ref{fromvectortoelement},
the polynomial has a solution in $\F_p$. Note that the reduction
can be computed in randomized polynomial time.
\end{proof}

\section{Complexity of the value set counting problem}

In this section, we prove several results about the
complexity of the  value set counting problem.

\subsection{Finite field extensions}

We will use a problem about $\NC^0_5$ circuits to prove that counting
the value set of a sparse polynomial in a binary field is \#P-hard.
A Boolean circuit is in $\NC^0_5$  if every output bit of the circuit
depends only on at most $5$ input bits. We can view a circuit with $n$
input bits and $m$ output bits as a map from $\{0,1\}^n $ to $\{0,1\}^m$
and call the image of the map the value set of the circuit.
The following proposition is implied in \cite{Durand94}.
We will sketch the proof for completeness.

\begin{proposition}\label{Durand}
Given a 3SAT formula with $n$ variables and $m$ clauses,
one can construct in polynomial time an $\NC^0_5$ circuit
with $n+m$ input bits and $n+m$ outputs bits, such
that if there are $M$  satisfying assignments for the
3SAT formula, then the cardinality of the value set of
the $\NC_5^0$ circuit is $2^{n+m} - 2^{m-1} M$.
In particular, if
the 3SAT formula can not be satisfied, then the circuit computes a
permutation from $\{0,1\}^{n+m}  $  to $\{0,1\}^{n+m}  $.
\end{proposition}

\begin{proof}
Denote the variables of the 3SAT formula by $x_1, x_2, \cdots, x_n$,
and the clauses of the 3SAT formula by $C_1, C_2, \cdots, C_m$.
Build a circuit with $n+m$ input bits and $n+m$ output bits as follows.
The input bits will be denoted by $x_1, x_2, \cdots, x_n, y_1, y_2, \cdots,
y_m$ and output bits will be denoted by $z_1, z_2, \cdots, z_n, w_1, w_2, \cdots,
w_m$. Set $z_i = x_i$ for $1\leq i \leq n$. And set
$$w_i = (C_i \wedge (y_i \oplus y_{(i+1 \pmod{m})})) \vee (\neg C_i \wedge y_i) $$
for $1\leq i\leq m$. In other words, if $C_i$ is evaluated to be TRUE,
then output $y_i \oplus y_{(i+1 \pmod{m})}$ as $w_i$,
and otherwise output $y_i $ as $w_i$. Note that $C_i$
depends only on $3$ variables
from $\{x_1,x_2, \cdots, x_n\}$, thus we obtain an $\NC_5^0$ circuit.
After fixing an assignment to $x_i$'s, $z_i$'s are also fixed, and
the transformation from
$(y_1, y_2, \cdots, y_m)$ to $(w_1, w_2, \cdots, w_m)$
is linear over $\F_2$.
One can verify that the linear transformation
has rank $m-1$ if the assignment satisfies all the clauses,
and it has rank $m$ (namely it is full rank)
if some of the clauses are not satisfied. So the cardinality
of the value set of the circuit is
$$ M 2^{m-1} + (2^n - M) 2^m = 2^{n+m} - 2^{m-1} M. $$
\end{proof}

If we replace the Boolean gates
in the $\NC^0_5$ circuit by  algebraic gates over $\F_2$, we obtain
an algebraic circuit that computes  a  polynomial map
from $\F_2^{n+m}$ to itself, where
 each polynomial depends only on
$5$ variables and has degree  equal to or less than $5$.
There is an $\F_2$-basis for $\F_{2^{n+m}}$, say $\omega_1, \omega_2, \cdots, \omega_{n+m}$
which induces a bijection from $\F_2^{n+m}  $ to $\F_{2^{n+m}}  $ given by
$$ (x_1,x_2,\cdots, x_{n+m}) \mapsto x= \sum_{i=1}^{n+m} x_i \omega_i  $$
which has an  inverse that can be represented by sparse polynomials
in $\F_{2^{n+m}}[x]$.
Using this fact,  we can replace the input bits of the algebraic
circuit by sparse polynomials,
and collect the output bits together using the base to form a single
element in $\F_{2^{n+m}}$. We thus obtain a sparse univariate polynomial
in $\F_{2^{n+m}}[x]$ from the $\NC_5^0$ circuit such that their value sets
have the same cardinality.  We thus have the following theorem:

\begin{theorem}\label{sparse}
Given a 3SAT formula with $n$ variables and $  m $ clauses,
one can construct in polynomial time a sparse polynomial
$\gamma(x)$ in $\F_{2^{n+m}}$
 such that the value set of $\gamma (x)$ has cardinality
$2^{n+m} - 2^{m-1} M$,
where $M$ is the number of satisfying assignments of the
3SAT formula.
\end{theorem}

Since counting the number of satisfying assignments for a 3SAT formula
is \#P-complete, we have our main theorem:

\begin{theorem}
The problem
of counting the value set of a sparse polynomial over a finite field of characteristic $p=2$
is \#P-hard.
\label{thm:binaryfieldvaluesetcounting}
\end{theorem}

\subsection{Prime-order finite fields}

The construction in Theorem~\ref{sparse} relies on building
field extensions. The technique cannot be adopted easily
to the prime-order finite field case.
We will prove that counting the value set of a straight-line polynomial
over prime-order finite field is \#P-hard. We reduce the counting version
of subset sum problem
to the value set counting problem.

\begin{theorem}\label{primesharpP}
Given access to an oracle that solves the value set counting problem for
straight-line polynomials over prime-order finite fields,
  there is a randomized polynomial-time algorithm  solving the
counting version of the \textsc{SSP}.
\end{theorem}

\begin{proof}
Given an instance of the counting subset sum problem, $b$ and
$S = \{ a_1, a_2, \cdots, a_n\} $,
if $ b > \sum_{i=1}^n a_i $, we answer $0$; if $b = 0$, then we answer $1$.
Otherwise  we find a prime $p > \max(2^{3t}, 2\sum_{i=1}^n a_i)  $
and ask the oracle to count the value set of the shift sparse polynomial
  $$ f(x):= (1- \beta(x)^{p-1} )  (\sum_{i=0}^{t-1}\alpha(x+i)2^i)$$
  over the prime-order field $\F_p$. We
output the answer $\card{V_f} -1$, which is easily seen to be exactly the number of
subsets of $\{ a_1, \cdots, a_n\}$ which sum to $b$.
\end{proof}

Since the counting version of the \textsc{SSP} is \#P-complete, this theorem yields
\begin{theorem}
Over a prime-order finite field $\F_p$,  the problem of counting the value set is \#P-hard under RP-reduction,
if the polynomial is given as a straight-line program.
\label{thm:primeorderfieldvaluesetcounting}
\end{theorem}

\section{The Image Set and Point Counting}

\begin{proposition}
If $f \in \mathbb{F}_q \left[ x \right]$ is a polynomial of degree $d>0$, then the cardinality of its image set is
\begin{equation}
\card{V_f} = \sum_{i=1}^d (-1)^{i-1} N_i \sigma_{i} \left( 1, \frac{1}{2}, \ldots, \frac{1}{d} \right)
\label{eq:pointsinimage}
\end{equation}
where $N_k = \card{\left\{ (x_1, \ldots, x_k) \in \mathbb{F}_q^k \mid f(x_1) = \cdots = f(x_k) \right\} }$ and $\sigma_{i}$ denotes the $i$th elementary symmetric function on $d$ elements.
\label{prop:pointstocard}
\end{proposition}

\begin{proof}
For any $y \in V_f$, define
\[\tilde{N}_{k,y} = \left\{ \left( x_1, \ldots, x_k \right) \in \F_q^k \mid f(x_1) = \cdots = f(x_k) = y \right\}\]
and denote the corresponding cardinality of these sets as
\[N_{k,y} = \card{ \tilde{N}_{k,y} } \]
and finally, note that
\begin{equation}
N_k = \sum_{y \in V_f} N_{k,y}
\label{eqn:newNk}
\end{equation}

Let us refer to the right hand side of (\ref{eq:pointsinimage}) as $\eta$; plugging (\ref{eqn:newNk}) into this expression and rearranging, we get
\begin{equation*}
\eta = \sum_{y \in V_f} \sum_{i=1}^d (-1)^{i-1} N_{i,y} \, \sigma_{i} \left( 1, \frac{1}{2}, \ldots, \frac{1}{d} \right).
\end{equation*}

Let us call the inner sum $\omega_y$, that is:
\[\omega_y = \sum_{i=1}^d (-1)^{i-1} N_{i,y} \, \sigma_{i} \left( 1, \frac{1}{2}, \ldots, \frac{1}{d} \right).\]

If we can show that for all $y \in V_f$ we have $\omega_y = 1$, then we clearly have $\eta = \card{V_f}$.

Let $y \in V_f$ be fixed.  Let $k=\card{f^{-1}(y)}$.  It is clear that $1\leq k \leq d$ and $N_{i,y} = k^i$ for $0\leq i \leq d$.
Substituting this in, our expression mercifully becomes somewhat nicer:

\begin{align}
\nonumber \omega_y &= 1 - \sum_{i=0}^d (-1)^{i} k^i \sigma_{i} \left( 1, \frac{1}{2}, \ldots, \frac{1}{d} \right)\\
\label{eqn:precancelterm} &= 1 - \sum_{i=0}^d (-1)^{i} \sigma_{i} \left( k 1, k \frac{1}{2}, \ldots, k \frac{1}{d} \right)\\
\label{eqn:cancelterm} &= 1 - \left[ \left( 1 - k 1 \right) \left( 1 - k \frac{1}{2} \right) \cdots \left( 1 - k \frac{1}{d} \right) \right]\\
\nonumber &= 1.
\end{align}

From step (\ref{eqn:precancelterm}) to step (\ref{eqn:cancelterm}), we are using the identity
\[\prod_{j=1}^n \left( \lambda - X_j \right) = \sum_{j=0}^n \left(-1 \right)^j \lambda^{n-j} \sigma_j \left(X_1, \ldots, X_n \right). \]

Note that the bracketed term of (\ref{eqn:cancelterm}) is $0$, as $k$ must be an integer such that $1 \leq k \leq d$, so one term in the product will be $0$.

Thus, we have $\eta = \card{V_f}$, as desired.
\end{proof}

Proposition \ref{prop:pointstocard} gives us a way to express $\card{V_f}$ in terms of the numbers of rational points on a sequence of curves over $\F_q$.
If we had a way of getting $N_k$ for $1 \leq k \leq d$, then it would be easy to calculate $\card{V_f}$.

The spaces $\tilde{N}_k$ aren't of any nice form (in particular, we cannot assume they are non-singular projective, abelian varieties, etc.), so we proceed by using the $p$-adic point counting method described in \cite{LauderWan2008countingpoints}, which works for any variety over a field of small characteristic (i.e., $p = O((d \log q)^C)$ for some positive constant $C$).

\begin{theorem}
There exists an explicit deterministic algorithm and an explicit polynomial $R$ such that for any $f \in \F_q [x]$ of degree $d$, where $q = p^m$ ($p$ prime), the algorithm computes the cardinality of the image set, $\card{V_f}$, in a number of bit operations bounded by $R \left( m^d d^d p^d \right)$.
\label{thm:imagepointcounting}
\end{theorem}

\begin{proof}
Recall that $N_k = \card{\tilde{N}_k}$ with
\begin{align*}
\tilde{N}_k & = \left\{ \left( x_1, \ldots, x_k \right) \in \F_q^k \mid f(x_1) = \cdots = f(x_k) \right\} \\
 &= \left\{ \left( x_1, \ldots, x_k \right) \in \F_q^k \left| \begin{array}{ccc}
    f(x_1) &- f(x_2) &= 0\\
    f(x_1) &- f(x_3) &= 0\\
    & \vdots &\\
    f(x_1) &- f(x_k) &= 0
    \end{array} \right. \right\}.
\end{align*}

For reasons soon to become clear, we need to represent this as the solution set of a single polynomial. Let us introduce additional variables $z_1$ to $z_{k-1}$, and denote $x = \left( x_1, \ldots, x_k \right)$ and $z = \left( z_1, \ldots, z_{k-1} \right)$.  Now examine the auxiliary function
\begin{equation}
F_k \left(x, z \right) = z_1 \left( f(x_1) - f(x_2) \right) + \cdots + z_{k-1} \left( f(x_1) - f(x_{k}) \right).
\label{eqn:auxeqn}
\end{equation}

Clearly, if $\gamma \in \tilde{N}_k$, then $F_k \left( \gamma, z \right)$ is the zero function.  If $\gamma \in \F_q^k \setminus \tilde{N}_k$, then the solutions of $F_k \left( \gamma, z \right) =0$ specify a $\left(k-2\right)$-dimensional $\F_q$-linear subspace of $\F_q^{k-1}$.
Thus, if we denote the cardinality of the solution set to $F_k(x,z) = 0$ as $\card{F_k}$, then we see that
\begin{align*}
\card{F_k} &= q^{k-1} N_{k} + q^{k-2} \left( q^k - N_k \right) \\
& = N_k q^{k-2} \left( q - 1 \right) + q^{2 k-2}.
\end{align*}

Solving for $N_k$, we find that
\begin{equation}
N_k = \frac{\card{F_k} - q^{2 k-2}}{q^{k-2} \left( q - 1 \right)}.
\label{eqn:Nkintermsofaux}
\end{equation}
Thus we have an easy way to determine what $N_k$ is depending on the number of points on this hypersurface defined by the single polynomial equation
$F_k=0$.

The main theorem in \cite{LauderWan2008countingpoints} yields an algorithm for toric point counting in $\F_{q^\ell}$ for small characteristic (i.e., $p = O((d \log q)^C)$ for some positive constant $C$) that works for general varieties.  In \cite[\S 6.4]{LauderWan2008countingpoints}, this theorem is adapted to be a generic point counting algorithm.

Adapting this result to our problem, we see that $F_k$ has a total degree of $d+1$, is in $2 k -1$ variables, and that we only care about the case where $\ell = 1$. Thus, the runtime for this algorithm is $\tilde{O} ( 2^{8k + 1} m^{6k + 4} k^{6 k + 2} d^{6 k - 3} p^{4 k + 2} )$ bit operations. In order to calculate $\card{V_f}$ using equation (\ref{eq:pointsinimage}), we calculate $N_k$ for $1 \leq k \leq d$, scaled by an elementary symmetric polynomial.  All of the necessary elementary symmetric polynomials can be evaluated using Newton's identity (see \cite{mead1992}) in less than 
$O(d^2 \log d)$ multiplications.  As such, the entire calculation has a runtime of $\tilde{O} ( 2^{8d + 1} m^{6d + 4} d^{12 d - 1} p^{4 d + 2} )$ bit operations.  For consistency with \cite{LauderWan2008countingpoints}, we can then note that as $d>1$, we can write $2^{8d + 1} = d^{\left( \log_d 2 \right) \left( 8d + 1 \right)}$.  Thus, there is a polynomial,  $R$, in one variable such that the runtime of this algorithm is bounded by $R ( m^d d^d p^d )$ bit operations.  In the dense polynomial model, the polynomial $f$ has input size $O \left( d \log q \right)$, so this algorithm does not have polynomial runtime with respect to the input length.  This algorithm has runtime that is exponential in the degree of the polynomial, $d$, and polynomial in $m$ and $p$.
\end{proof}

\section{Open Problems}

Though  value sets of polynomials appear to be closely related to zero sets,
they are not as well-studied. There are many interesting open problems about
value sets. The most important one is to find a counting algorithm with
running time $(d\log q)^{O(1)}$, that is, a deterministic polynomial time algorithm in the dense model.  It is not clear if this is always possible.
Our result affirmatively solves this problem for fixed $d$
if characteristic $p$ is reasonably small. We conjecture that the same result is true for fixed $d$
and all characteristic $p$.

For the complexity side, can one prove that the counting problem for sparse polynomials
in prime-order finite fields is hard? Can one prove that the counting problem for dense input
model is hard for general degree $d$?

\paragraph{Acknowledgment:} We thank Dr. Tsuyoshi Ito for pointing
out the reference \cite{Durand94} to us.

\bibliographystyle{plain}
\bibliography{vs}

\end{document}